\documentclass[11pt]{article}
\date{}

\usepackage[title]{appendix}
\usepackage{xcolor}
\usepackage[margin=1in]{geometry}                % See geometry.pdf to learn the layout options. There are lots.
\usepackage{graphicx}
\usepackage{subcaption}
\usepackage{pdflscape}
\usepackage{amssymb}
\usepackage[normalem]{ulem}
\usepackage{hyperref}
\usepackage{enumitem}
\usepackage{mathrsfs}
\usepackage{epstopdf}
\usepackage{rotating}
\usepackage{longtable} % for 'longtable' environment
\usepackage{adjustbox}
\usepackage{float}

\usepackage{color}
\usepackage{bbm, dsfont}
\usepackage{pst-node}
\usepackage{tikz-cd}
\usepackage{amsfonts} %mathbb{}
\usepackage{geometry}
\usepackage{amsthm}
\usepackage{amsmath}
\usepackage{titlesec}
\usepackage{amssymb}
\usepackage{enumitem}
\usepackage{float}
\usepackage [english]{babel}
\usepackage [autostyle, english = american]{csquotes}
\usepackage{algorithm}
\usepackage[noend]{algpseudocode} %pseudocode
\makeatletter
\def\BState{\State\hskip-\ALG@thistlm}
\makeatother
\usepackage{hyperref}
\usepackage[normalem]{ulem}
\usepackage{mathrsfs}
\usepackage[italicdiff]{physics}

\newlist{casess}{enumerate}{1}
\setlist[casess]{label=     \textbf{Case} \arabic*:}
\usepackage{mathtools}

\makeatletter
\newcommand*{\rom}[1]{\expandafter\@slowromancap\romannumeral #1@}
\makeatother

\usepackage{etoolbox}

\makeatletter
\patchcmd{\ttlh@hang}{\parindent\z@}{\parindent\z@\leavevmode}{}{}
\patchcmd{\ttlh@hang}{\noindent}{}{}{}
\makeatother

\usepackage{listings}
\usepackage{color} %red, green, blue, yellow, cyan, magenta, black, white
\definecolor{mygreen}{RGB}{28,172,0} % color values Red, Green, Blue
\definecolor{mylilas}{RGB}{170,55,241}

\newlist{Assumptions}{enumerate}{1}
\setlist[Assumptions]{label=     \textbf{Assumption} \arabic*:}

\makeatletter

\newsavebox{\@brx}
\newcommand{\llangle}[1][]{\savebox{\@brx}{\(\m@th{#1\langle}\)}%
  \mathopen{\copy\@brx\kern-0.5\wd\@brx\usebox{\@brx}}}
\newcommand{\rrangle}[1][]{\savebox{\@brx}{\(\m@th{#1\rangle}\)}%
  \mathclose{\copy\@brx\kern-0.5\wd\@brx\usebox{\@brx}}}
\makeatother

\usepackage{lipsum} % for generating filler text
\usepackage{titlesec}
\titleformat{\subsection}[runin]% runin puts it in the same paragraph
       {\normalfont\bfseries}% formatting commands to apply to the whole heading
       {\thesubsection}% the label and number
       {0.5em}% space between label/number and subsection title
       {}% formatting commands applied just to subsection title
       [.]% punctuation or other commands following subsection title

 \newtheorem{thm}{Theorem}[section]
 \newtheorem{cor}[thm]{Corollary}
 
 \newtheorem{lem}[thm]{Lemma}
 
 \theoremstyle{definition}
 \newtheorem{defn}[thm]{Definition}
 \theoremstyle{remark}
 \newtheorem{rem}[thm]{Remark}

 \numberwithin{equation}{section}

\numberwithin{equation}{section}

\def\N{\mathbb{N}}

\def\Z{\mathbb{Z}}

\def\Z{\mathbb Z}

\DeclarePairedDelimiterX{\inp}[2]{\langle}{\rangle}{#1, #2}

\makeatletter
\newcommand*\bigcdot{\mathpalette\bigcdot@{.5}}
\newcommand*\bigcdot@[2]{\mathbin{\vcenter{\hbox{\scalebox{#2}{$\m@th#1\bullet$}}}}}
\makeatother

\def\<{\langle}
\def\>{\rangle}

\numberwithin{equation}{section}

\usepackage[backend=biber,maxnames=10,style=alphabetic]{biblatex}
\begin{document}

\title{Extreme proximal actions and primitivity of group rings}

\author{Felipe I. Flores
}

\maketitle

\begin{abstract}\setlength{\parindent}{0pt}\setlength{\parskip}{1ex}\noindent
We present a purely dynamical criterion for the primitivity of group rings. To be precise, we show that if $G$ is a countable group that admits a topologically free, extremely proximal action, then the group ring $KG$ is primitive for any field $K$. Our criterion recovers several of the existing results on primitivity and applies to new examples, such as the commensurator groups of torsion-free non-elementary hyperbolic groups, some Kac-Moody groups, and other weakly hyperbolic groups.

\end{abstract}

%%%%%%%%%%%%%%%%%%%%%%%%%%%%%%%%%%%%%%%%%%%%%%%%%%%%%%%%%%%%%%%%%%%%%%%%%%%%%%%%%%%%%%
\section{Introduction}
%%%%%%%%%%%%%%%%%%%%%%%%%%%%%%%%%%%%%%%%%%%%%%%%%%%%%%%%%%%%%%%%%%%%%%%%%%%%%%%%%%%%%%

A ring $R$ is said to be (right) primitive if it has a faithful irreducible (right) $R$-module. For many years, it was not known whether there exists a group $G$ and a ring $R$ such that the group ring $RG$ is primitive \cite[Problem 17]{Ka70}. However, Formanek and Snider provided the first example of such a group in \cite{FoSn72}, and they even showed that any group can be embedded in another group $G$ such that $KG$ is primitive for some field $K$.

Since then, the study of primitivity in group rings has been an active area of research for several decades. A major milestone was achieved in 1978 through the work of Domanov \cite{Do78}, Farkas–Passman \cite{FaPa78}, and Roseblade \cite{Ro78,Ro79}, who established a complete characterization of the primitivity of group rings of polycyclic-by-finite groups. More concretely, they proved that, for a polycyclic-by-finite group $G$, the group ring $KG$ is primitive if and only if the FC-center of $G$ is trivial and the field $K$ is not absolute.

In the setting of non-noetherian groups, we can mention the results on non-elementary free products by Formanek \cite{Fo73}, the results on amalgamated free products by Balogun \cite{Bo89}, and those by Alexander-Nishinaka \cite{AlNi17}. Very recently, using geometric methods, Solie established the primitivity of all group rings $KG$ associated with torsion-free (non-elementary) hyperbolic groups \cite{So18}, and his result was extended to all countable acylindrically hyperbolic groups without finite normal subgroups by Abbott-Dahmani \cite{AbDa19}.

Other interesting results on the primitivity of group rings can be found in the works of Nishinaka \cite{Ni07,Ni11}.

As we mentioned in the abstract, the goal of this article is to introduce a purely dynamical criterion that implies the primitivity of group rings associated with the acting group. This will be achieved using the concept of extremely proximal actions. We introduce the appropriate definitions now.

\begin{defn}\label{thedef}
Let $G$ be a group, and let $X$ be a Hausdorff space with $\# X>2$. An action by homeomorphisms $G\curvearrowright X$ is said to be 
\begin{enumerate}
    \item an \emph{extremely proximal action} if, for every pair of non-empty open subsets $U,V\subseteq X$, there exists $g\in G$ such that $g(X\setminus U)\subseteq V$;
    \item \emph{topologically free} if, for every $g\in G\setminus\{1\}$, the set of points fixed by $g$ is nowhere dense.
\end{enumerate}
\end{defn}

The class of groups admitting a topologically free extremely proximal action contains all acylindrically hyperbolic groups without finite normal subgroups \cite{AbDa19,Yan25}, commensurator groups of non-elementary torsion-free hyperbolic groups \cite{BaFl26}, and lim-free weakly hyperbolic groups \cite{Ry26}. This last class of groups contains examples as interesting as the Burger-Mozes group, some Kac-Moody groups, the Amir-Lazarovich groups, and most Baumslag-Solitar groups. More examples, coming from the theory of groups acting on trees, can be found in the works \cite{BrIvOm20,FLMS}.\footnote{Note that some of these actions are not explicitly studied as extremely proximal actions, but rather as actions with `north-south dynamics.' The standard argument showing that they are extremely proximal can be found in \cite[Lemma 2.4]{BaFl26}.}

Let us also mention that, due to their striking role in the theory of $C^*$-algebras \cite{Oz25,FKOCP26}, the study of these actions is currently receiving a lot of attention, and it is likely that we will see new examples soon.

The main result of this article is as follows:

\begin{thm}\label{mainthm}
   Let $G$ be a countable group that admits a topologically free, extremely proximal action $G\curvearrowright X$. Then, if $R$ is a countable domain, the group ring $RG$ is primitive. Moreover, the group ring $KG$ is primitive for any field $K$.
\end{thm}

Since non-elementary acylindrically hyperbolic groups without finite normal subgroups have topologically free extremely proximal actions \cite[Theorem 1.3]{Yan25}, our main result generalizes the primitivity result of Abbott-Dahmani \cite[Theorem 0.4]{AbDa19}. In fact, our approach is inspired by theirs.

The present article replaces the author's paper `Mixed-identity-freeness and primitivity of group rings,' which previously circulated as a preprint. See Remark \ref{observation}.

%%%%%%%%%%%%%%%%%%%%%%%%%%%%%%%%%%%%%%%%%%%%%%%%%%%%%%%%%%%%%%%%%%%%%%%%%%%%%%%%%%%%%%
\section{Proof of the main result}
%%%%%%%%%%%%%%%%%%%%%%%%%%%%%%%%%%%%%%%%%%%%%%%%%%%%%%%%%%%%%%%%%%%%%%%%%%%%%%%%%%%%%%

Before beginning the proof of Theorem \ref{mainthm}, we recall the following property introduced by Alexander and Nishinaka \cite{AlNi17}.

\begin{quote} $(**)$ For each finite subset $M\subset G\setminus\{1\}$, there exist distinct $a, b, c \in G$ such that if $(x^{-1}_1g_1x_1)(x^{-1}_2g_2x_2)\cdots(x^{-1}_m g_mx_m) = 1,$ where $m \geq 2$, $g_i \in M$, and
$x_i \in \{a, b, c\}$ for all $i = 1, \ldots , m$, then $x_i = x_{i+1}$ for some $i$.
\end{quote}

\begin{thm}[Alexander-Nishinaka]\label{AN}
Let $G$ be a group that has a non-abelian free subgroup whose cardinality is the same as that of $G$, and suppose that $G$ satisfies property $(**)$. Then, if $R$ is a domain with $|R| \leq |G|$, the group ring $RG$ of $G$ over $R$ is primitive. Moreover, the group ring $KG$ is primitive for any field $K$.
\end{thm}

Note that, in order to derive Theorem \ref{mainthm}, we only need to prove that groups admitting a topologically free, extremely proximal action contain non-abelian free subgroups and satisfy property $(**)$.

\begin{rem}\label{observation}
In \cite[p. 222]{AlNi17}, Alexander and Nishinaka state the following property instead of $(**)$.
\begin{quote} $(*)$ For each finite subset $M\subset G\setminus\{1\}$ and for any positive integer $m \geq 2$, there exist distinct $a, b, c \in G$ such that if $(x^{-1}_1g_1x_1)(x^{-1}_2g_2x_2)\cdots(x^{-1}_m g_mx_m) = 1$, where $g_i \in M$ and
$x_i \in \{a, b, c\}$ for all $i = 1, \ldots , m$, then $x_i = x_{i+1}$ for some $i$.
\end{quote}
In their main theorem \cite[Theorem 1.1]{AlNi17}, Alexander and Nishinaka write that $(*)$ is enough to obtain the primitivity conclusion displayed in Theorem \ref{AN}. However, a careful examination of their article reveals that, after \cite[Remark 3.1]{AlNi17}, they actually prove their results under the formally stronger property $(**)$.

In a previous version of this article, entitled `Mixed-identity-freeness and primitivity of group rings,' the author observed that mixed-identity free groups satisfy $(*)$, which initially led him to state a theorem that the current technology does not allow him to prove. This manuscript recovers the dynamical theorem previously proved there.

The printed proofs of \cite[Theorem 3.2]{AbDa19} and \cite[Theorem 6]{So18} also verify property $(*)$ instead of $(**)$, making them apparently formally incorrect. We note, however, that the corresponding `big powers' lemmas \cite[Lemma 3.3]{AbDa19} and \cite[Theorem 4]{So18} are strong enough to verify $(**)$, so their main theorems are ultimately correct. We follow a similar strategy here.
\end{rem}

The following lemma is our `big powers' lemma.

\begin{lem}\label{mainlemma}
Let $G$ be a group that admits a topologically free, extremely proximal action $G\curvearrowright X$. Given a finite subset $M\subset G\setminus\{1\}$, there exists an element $u\in G$ of infinite order such that, for all elements $g_1,\dots, g_k\in M$ and all $n_0,\dots, n_k\in \Z\setminus\{0\}$, one has
\[
u^{n_0}g_1u^{n_1}g_2\cdots u^{n_{k-1}}g_ku^{n_k}\neq1.
\] 
\end{lem}

\begin{proof}
Since $G\curvearrowright X$ is topologically free, there is a point $x_0\in X$ such that $g_ix_0\not= x_0$ for all $1\leq i\leq k$. So, by appealing to Hausdorffness, we can also find open sets $V_0,\ldots,V_k$ such that $V_0\cap V_i=\emptyset$ and $x_0 \in V_0, g_ix_0 \in V_i$ for all $1\leq i\leq k$. We set
   \[
      U = V_0 \cap \bigcap_{i=1}^k g_i^{-1}V_i,
   \]
   and note that $x_0\in U$ and $U \cap g_iU = \emptyset$ for all $1\leq i\leq k$.
   
  Now, since  $U \setminus \{x_0\}$ is a non-empty\footnote{Note that the extremely proximal assumption forces $X$ to lack isolated points, hence $U$ is necessarily infinite.} open subset of $X$, there are disjoint open subsets $ U^{\pm} \subset U \setminus \{x_0\} $ and a group element $u\in G$ such that
    \[
      u(X \setminus U^-) \subset U^+\quad\text{ and } \quad u^{-1}(X \setminus U^+) \subset U^-.
   \]
   
   Finally, we check that the group element $w(u):=u^{n_0}g_1u^{n_1}g_2\cdots u^{n_{k-1}}g_ku^{n_k}$ cannot be trivial. Indeed, we have
   \begin{align*}
      w(u)x_0  &= (u^{n_0}g_1u^{n_1}g_2\cdots u^{n_{k-1}}g_ku^{n_k})x_0 \\
      &\in (u^{n_0}g_1u^{n_1}g_2\cdots u^{n_{k-1}}g_k)(U^+\cup U^-) \\
      &\subset (u^{n_0}g_1u^{n_1}g_2\cdots u^{n_{k-1}}) (X \setminus U) \\
		&\subset (u^{n_0}g_1 u^{n_1}g_2 \cdots u^{n_{k-2}}g_{k-1}) (U^+ \cup U^-) \\
        &\subset (u^{n_0}g_1 u^{n_1}g_2 \cdots u^{n_{k-2}}) (X\setminus U) \\
        &\vdots \\
		&\subset u^{n_0}g_1 (U) \subset U^+ \cup U^-.
   \end{align*}
   The above computation leads to the conclusion $w(u)x_0\not=x_0$, since $x_0 \notin U^+ \cup U^-$. Furthermore, we also see that $u^nx_0\in U^+$, for all $n>0$. This implies that $u$ has infinite order and finishes the proof.
\end{proof}

\begin{lem}[Glasner]\label{GL}
    Let $G$ be a group that admits an extremely proximal action $G\curvearrowright X$. Then, $G$ contains a non-abelian free subgroup
\end{lem}
\begin{proof}
    We can find a non-abelian free subgroup by arguing exactly as in the proof of \cite[Theorem 3.4]{Gl74}.
\end{proof}

\begin{proof}[Proof of Theorem \ref{mainthm}]
Let $M$ be a finite collection of non-trivial elements of $G\setminus\{1\}$, and let $u\in G$ be the group element obtained from applying Lemma \ref{mainlemma}.  

Let $g_1,\dots, g_m\in M$, and consider a word 
\[
w=(x_1^{-1}g_1x_1)(x_2^{-1}g_2x_2)\cdots (x_m^{-1}g_mx_m),
\]
where $x_i\in \{u,u^{2},u^{3}\}$ for each $i=1,\dots,m$. Then $w$ becomes
\[
w=u^{n_0}g_1u^{n_1}g_2\cdots u^{n_{m-1}}g_mu^{n_m},
\]
where $u^{n_0}=x_1^{-1}$, $u^{n_i}=x_ix_{i+1}^{-1}$ for $i=1,\dots, m-1$, and $u^{n_m}=x_m$.  By construction, $n_i\in \{0,\pm 1, \pm 2\}$ for each $i=1,\dots, m-1$.  If $w=1$, then by Lemma \ref{mainlemma}, we must have $n_i=0$ for some $i\in \{1,\dots, m-1\}$.  Thus $x_ix_{i+1}^{-1}=1$.  Furthermore, the elements $a=u$, $b=u^{2}$, and $c=u^{3}$ are distinct and so it follows that $G$ satisfies property $(**)$. The conclusion now follows from Theorem \ref{AN}. \end{proof}

As an immediate consequence of our results, we also note that groups admitting a topologically free extremely proximal action are mixed-identity-free (MIF). Particular versions of this result were previously obtained by Ozawa \cite{Oz25} and Rybak \cite{Ry26}.

\begin{defn}\label{MIF}
    A group $G$ is mixed-identity-free (MIF) if, for every $n\in \N$ and every $1\neq w\in G*\mathbb F_n$, there is a homomorphism $G*\mathbb F_n\to G$, restricting to the identity on $G$, that does not send $w$ to $1$.
\end{defn}

\begin{cor}
    Let $G$ be a group that admits a topologically free, extremely proximal action $G\curvearrowright X$. Then, $G$ is MIF.
\end{cor}

\begin{proof}
    Since $G$ contains a non-abelian free subgroup, it does not satisfy any identity, and so by \cite[Remark 5.1]{HuOs16}, $G$ is mixed-identity-free if and only if it does not satisfy a non-trivial mixed identity $ w \in G * \left\langle {t} \right\rangle $. We now show that this is indeed the case, assuming without loss of generality that $ w \notin \left\langle {t} \right\rangle $. Then, up to conjugation, we may decompose $w$ as
   $$
      w = t^{n_1}g_1 t^{n_2}g_2 \cdots t^{n_k}g_{k},
   $$
   where $ g_1, \dots, g_{k} \in G \setminus \{1\} $ and $ n_1, \dots, n_k \in \mathbbm{Z} \setminus \{0\} $. Following the same argument given in the proof of Lemma \ref{mainlemma}, we can find $u\in G$ such that $w(u)\not=1$.
\end{proof}

%%%%%%%%%%%%%%%%%%%%%%%%%%%%%%%%%%%%%%%%%%%%%%%%%%%%%%%%%%%%%%%%%%%%%%%%%%%
\section*{AI statement}
%%%%%%%%%%%%%%%%%%%%%%%%%%%%%%%%%%%%%%%%%%%%%%%%%%%%%%%%%%%%%%%%%%%%%%%%%%%

No generative AI or large language model was used in the production of this text or its mathematical content.

%%%%%%%%%%%%%%%%%%%%%%%%%%%%%%%%%%%%%%%%%%%%%%%%%%%%%%%%%%%%%%%%%%%%%%%%%%%
\section*{Acknowledgments}
%%%%%%%%%%%%%%%%%%%%%%%%%%%%%%%%%%%%%%%%%%%%%%%%%%%%%%%%%%%%%%%%%%%%%%%%%%%

The author gratefully acknowledges support from the Simons Foundation Dissertation Fellowship SFI-MPS-SDF-00015100. He also thanks Ben Hayes, Aaratrick Basu, and Soham Chakraborty for all the interesting discussions surrounding the present topic. The author also thanks the anonymous referee who observed the error reported in Remark \ref{observation}.

\printbibliography

\bigskip
\bigskip
ADDRESS

\smallskip
\smallskip
Felipe I. Flores

Department of Mathematics, University of Virginia,

114 Kerchof Hall. 141 Cabell Dr,

Charlottesville, Virginia, United States

E-mail: hmy3tf@virginia.edu
\end{document}